\documentclass[12pt]{article}
\usepackage{amsmath,amssymb,amsthm}
\usepackage{url}
\usepackage[arrow,matrix,curve]{xy}
\newtheorem{cor}{Corollary}
\newtheorem*{thm*}{Theorem}
\newtheorem*{prop*}{Proposition}
\newtheorem*{lem*}{Lemma}
\newtheorem*{cor*}{Corollary}
\theoremstyle{definition}

\theoremstyle{definition}

\theoremstyle{definition}
\newtheorem*{rmk*}{Remark}
\theoremstyle{definition}

\newcommand{\ZZ}{\mathbb Z}
\newcommand{\RR}{\mathbb R}
\newcommand{\CC}{\mathbb C}
\newcommand{\HH}{\mathbb H}

\begin{document}

\author{Hanno von Bodecker\footnote{Fakult{\"a}t f{\"u}r Mathematik, Universit{\"a}t Bielefeld, 33501 Bielefeld, Germany}}

\title{A note on the double quaternionic transfer and its $f$--invariant}

\date{}

\maketitle

\begin{abstract}
It is well-known that for a line bundle over a closed framed manifold, its sphere bundle can also be given the structure of a framed manifold, usually referred to as a transfer. Given a pair of lines, the procedure can be generalized to obtain a double transfer. We study the quaternionic case, and derive a simple formula for the $f$--invariant of the underlying bordism class, enabling us to investigate its status in the Adams--Novikov spectral sequence. 
As an application, we treat the situation of quaternionic flag manifolds.
\end{abstract}

\section{Introduction and statement of the results}

It is well-known that the Pontrjagin--Thom construction identifies the stable stems with the framed bordism groups, i.e.\ $\pi^{s}_{*}S^{0}\cong\Omega_{*}^{fr}$. 
Starting from a real, complex, or quaternionic line bundle over a closed framed manifold, L\"offler and Smith \cite{Loffler:1974qf} considered the corresponding principal sphere bundles to obtain morphisms 
\begin{equation}\label{s_k}
S_{\mathbb{K}}\colon\Omega^{fr}_*\left(\mathbb{KP}^{\infty}\right)\rightarrow\Omega^{fr}_{*+\dim\mathbb{K}-1};
\end{equation}
in particular, if $\mathbb{K}=\CC$, Adams' classical $e$--invariant \cite{Adams:1966ys}
\begin{equation}\label{complex e}
e_{\mathbb{C}}\colon\pi^{s}_{2k+1}S^{0}\rightarrow \mathbb{Q/Z}
\end{equation}
of such a transfer can be computed on the base by means of a simple cohomological formula \cite[Proposition 2.1]{Loffler:1974qf}, and similarly for $\mathbb{K}=\HH$ (although there is a slight mistake in \cite[Proposition 5.1]{Loffler:1974qf}, see \eqref{corrected e}). 

In \cite{Knapp:1979vn}, Knapp has given an extensive treatment of n-fold iterations of the map \eqref{s_k}  (and  variations thereof) for $\mathbb{K}=\CC$ using $K$--theory, and these techniques also prove useful for studying framings on Lie groups \cite{Knapp:1978zv}. Furthermore, iterated complex transfers have been related to framed hypersurface representations of certain beta elements at odd primes \cite{Carlisle:1985ve,Baker:1988fy}. In the more recent work  \cite{Imaoka:2007}, an analogous result has been obtained for the element $\nu^{*}\in\pi^{s}_{18}S^{0}$ using a double quaternionic transfer.

Much insight into the stable homotopy groups of the spheres can be gained by looking at the Adams--Novikov spectral sequence (ANSS), 
$$E_{2}^{p,q}\left[MU\right]=\textrm{Ext}^{p,q}_{MU_*MU}\left(MU_*,MU_*\right)\Rightarrow\pi_{q-p}^{s}S^{0}$$
and the rich algebraic structure inherent in complex oriented cohomolgy theories (see e.g.\ \cite{Ravenel:2004xh}); in particular,  the $E_{2}$ term of its $BP$--based analog is populated by the so-called greek letter elements. Within this framework, the $e$--invariant becomes an invariant of first filtration, since it factors through the 1--line of the ANSS.

In order to detect second filtration phenomena, Laures introduced the $f$--invariant, which 
is a follow-up to the $e$--invariant and takes values in the divided congruences between modular forms \cite{Laures:1999sh,Laures:2000bs}:
\begin{equation}\label{f}
f\colon\pi^{s}_{2k}S^{0}\rightarrow\underline{\underline{D}}^{\Gamma}_{k+1}\otimes{\mathbb{Q/Z}}
\end{equation}
Recent work has deepened our understanding of the algebraic part of the $f$--invariant, i.e.\ the map translating the 2--line into congruences \cite{Behrens:2008fp}, and the image of the beta family at $p=2$ is known in closed form \cite{Bodecker:2009fk}.

Similar to the $e$--invariant, which may be interpreted as the reduction of the relative Todd genus of a $(U,fr)$--manifold \cite{Conner:1966jw}, the $f$--invariant can be understood geometrically: This time,  it is the (suitably adapted relative version of the) Hirzebruch elliptic genus of a $(U,fr)^{2}$--manifold that reduces to the $f$--invariant of its corner of codimension two \cite{Laures:2000bs}, and this description is accessible to an index theoretical interpretation \cite{Bodecker:2008pi}; in particular, we could derive a cohomological formula for the $f$--invariant of the double complex transfer and perform some explicit calculations \cite[Section 5.1]{Bodecker:2008pi}.

The purpose of this note is to extend these results to the quaternionic situation: Given two quaternionic lines $\lambda_{\HH}$, $\lambda_{\HH}'$ over a closed framed manifold $B$, we can restrict the Whitney sum to the disks, resulting in a manifold $Z=D(\lambda_{\HH}')\oplus D(\lambda_{\HH})$, the codimension two corner of which is  $S(\lambda_{\HH}')\oplus S(\lambda_{\HH})$. The vertical tangent bundle of the latter can be trivialized using quaternion multiplication, which, when combined with the pullback of the tangent bundle of the base, can then be used to define a framing of this corner; passing to the underlying framed bordism classes, we obtain a morphism
\begin{equation}
S^{\oplus2}_{\HH}\colon\Omega^{fr}_{*}\left(\HH P^{\infty}\times\HH P^{\infty}\right)\rightarrow\Omega^{fr}_{*+6}.
\end{equation}
Clearly, if this construction is applied to a single point, we recover the non-trivial element $\nu^{2}\in\pi^{s}_{6}S^{0}\cong\ZZ/2$ (i.e.\ the permanent cycle $\alpha_{2/2}^{2}=\beta_{2/2}$); regarding positive-dimensional base spaces, we have:

\begin{thm*}\label{main_result}
Let $\lambda_{\HH}$, $\lambda_{\HH}'$ be quaternionic lines over a framed smooth closed manifold $B$ 
 of dimension $2m>0$ and let $\eta,\ \omega\in H^4\left(B,{\mathbb{Z}}\right)$ be the corresponding second Chern classes. Then, for arbitrary levels $N>1$, the $f$--invariant of the double transfer $S^{\oplus2}_{\HH}B$ is represented by
\begin{equation}\label{the formula}
f\left[S^{\oplus2}_{\HH}B\right]\equiv\left(-1\right)^{n+1}\sum_{k=1}^{n-1}\frac{B_{2k+2}}{k+1}G_{2n-2k+2}(\tau)\left\langle \frac{\eta^k}{\left(2k\right)!}\frac{\omega^{n-k}}{\left(2n-2k\right)!},[B]\right\rangle
\end{equation}
if $m=2n$, and it vanishes if $m\neq2n$.
\end{thm*}
\begin{rmk*} As a rather obvious consequence, the double quaternionic transfer necessarily misses several elements in the Adams--Novikov 2--line; in particular, at the prime two, this applies to the elements $\beta_{i/j}$ if $i$ is even but $j$ is odd, and to the permanent cycles $\alpha_{1}\bar\alpha_{4k}$ (which correspond to Im$J_{8k}$). Although not for dimensional reasons, this is also true for the permanent cycles $\alpha_{1}\alpha_{4k+1}$: These correspond to members of Adams' $\mu$--family \cite{Adams:1966ys}, which can be detected by the $d_{\RR}$--invariant, so they cannot bound a Spin manifold.
\end{rmk*}

As an application, we perform explicit calculations in some (admittedly low-dimensional) range, resulting in:

\begin{cor}\label{non-trivial examples}
Applying the transfers $S^{\oplus2}_{\HH}$ w.r.t.\ the tautological quaternionic lines, we have{\textup{:}}
\begin{itemize}
\item[\textup{(i)}] $S^{\oplus2}_{\HH}\left(Sp\left(2\right)/Sp\left(1\right)^{\times2}\right)^{\times2}$ corresponds to $\beta_{4/4}$ at $p=2$,
\item[\textup{(ii)}] $S^{\oplus2}_{\HH}\left(Sp\left(3\right)/Sp\left(1\right)^{\times3}\right)$ corresponds to $\pm\beta_{4/2,2}$ at $p=2$.
\end{itemize}
Moreover, modulo higher Adams--Novikov filtration, these classes generate the image of the transfer $S^{\oplus2}_{\HH}$ in $\pi^{s}_{6<2k\leq20}S^{0}$.
\end{cor}

\begin{rmk*}
It is well-known that $\pi^{s}_{18}S^{0}\cong\ZZ/8\oplus\ZZ/2$, which is reflected in the ANSS by the non-trivial group extension $4\beta_{4/2,2}=\alpha_{1}^{2}\beta_{4/3}$ (whereas the second summand is generated by a member of Adams' $\mu$--family, hence corresponding to $\alpha_{1}\alpha_{9}$); thus, Corollary \ref{non-trivial examples} checks nicely with Imaoka's result \cite{Imaoka:2007} that the element $\nu^{*}$ is in the image of a  double quaternionic transfer. 
\end{rmk*}

Unfortunately, we have to acknowledge the fact that the examples given above exhaust the list of interesting double  transfers on quaternionic flag manifolds, at least in the following sense:

\begin{cor}\label{vanishing result}
Let $n\geq4$; then, for any pair of tautological lines, we have{\textup{:}}
$$f\left[S^{\oplus2}_{\HH}\left(Sp\left(n\right)/Sp\left(1\right)^{\times n}\right)\right]=0.$$
\end{cor}

Before presenting the proofs, we would like to point out the following:
\begin{rmk*}
Although for a complex line $\lambda_{\CC}$ the sum $\lambda_{\CC}\oplus\overline{\lambda}_{\CC}$ can be given the structure of a quaternionic line, such `split' lines are far from sufficient to understand the transfer $S^{\oplus2}_{\HH}$; in particular,  neither $\beta_{4/4}$ nor $\beta_{4/2,2}$ can be obtained from a double transfer involving such a line (this is due to the improved divisibility results for the Chern numbers -- it {can}, however, be done for $2\beta_{4/2,2}=\beta_{4/2}$, e.g.\ by using a pair of these lines on the base  $G_{2}/T$).
\end{rmk*}



\section{Proof of the results}

\subsection{Recollection of notations and definitions}

Throughout this note, modular forms will be thought of in terms of their $q$--expansions (where $q=\exp(2\pi i \tau)$) at the cusp $i\infty$; our normalization conventions for the classical Eisenstein series can be summarized in terms of
\begin{equation*}
G_{2k}(\tau)=-\frac{B_{2k}}{4k}E_{2k}(\tau)=-\frac{B_{2k}}{4k}+\sum_{n\geq1}\sum_{d|n}d^{2k-1}\, q^{n},
\end{equation*}
and the argument will be omitted if confusion is unlikely. For later use, we remind the reader that the ring of modular forms w.r.t.\ $\Gamma_{1}(3)$ is generated by the odd Eisenstein series of weight one and three, viz.\
\begin{equation*}
E_1^{\Gamma_{1}(3)}=1+6\sum_{n\geq1}\sum_{d|n}(\frac{d}{3})\, q^n\mbox{  and  }E_3^{\Gamma_{1}(3)}=1-9\sum_{n\geq1}\sum_{d|n}(\frac{d}{3})d^2\, q^n,
\end{equation*}
where $(\frac{\cdot}{\cdot})$ is the Legendre symbol; again, whenever confusion is unlikely, we drop the superscript from the notation.

Regarding the definition of the $f$--invariant, let us briefly recall some notation from \cite{Laures:1999sh,Laures:2000bs}: Considering the congruence subgroup $\Gamma=\Gamma_1(N)$ for a fixed level $N$, let $\mathbb{Z}^{\Gamma}=\mathbb{Z}[\zeta_N,1/N]$ and denote by $M^{\Gamma}_*$ the graded ring of modular forms w.r.t.~$\Gamma$ which expand integrally, i.e.\ which lie in $\mathbb{Z}^{\Gamma}[\![q]\!]$. The ring of  divided congruences $D^{\Gamma}$  consists of those rational combinations of modular forms which expand integrally; this ring can be filtered by setting
\begin{equation*}
D_k^{\Gamma}=\left\{\left.f={\textstyle{\sum_{i=0}^{k}}}f_i\ \right| f_i\in M_i^{\Gamma}\otimes\mathbb{Q},\ f\in\mathbb{Z}^{\Gamma}[\![q]\!]\right\}.
\end{equation*}
Furthermore, put
\begin{equation*}\underline{\underline{D}}^{\Gamma}_{k}=D^{\Gamma}_k+M_0^{\Gamma}\otimes\mathbb{Q}+M_k^{\Gamma}\otimes\mathbb{Q}.
\end{equation*}
For our purposes, a $(U,fr)^{2}$--manifold is just a stably almost complex $\langle2\rangle$--manifold $Z$  together with a stable decomposition $TZ\cong E_{1}\oplus E_{2}$ (in terms of complex vector bundles)  and specified trivializations of  the restriction of the $E_{i}$ to the respective face $\partial_{i}Z$ for $i=1,2$ (strictly speaking, this defines the tangential version of such a structure, but it can be converted into the normal version once an embedding is chosen). Suppressing the induced structure from the notation, let $M$ denote the corner of a $2n+2$--dimensional $(U,fr)^{2}$--manifold $Z$; then the $f$--invariant of the underlying bordism class is defined to be
\begin{equation}\label{f using bordism}
f\left[M\right]\equiv\langle(Ell^{\Gamma}(E_1)-1)(Ell_{0}^{\Gamma}(E_2)-1),[Z,\partial Z]\rangle\mod\underline{\underline{D}}^{\Gamma}_{n+1},
\end{equation}
where $Ell^{\Gamma}$ is the Hirzebruch genus associated to $\Gamma$ and $Ell_{0}^{\Gamma}$ is its constant term in the $q$--expansion (i.e.\ the stable $\chi_{-\zeta}$ genus). For further details, we refer to \cite{Laures:2000bs,Bodecker:2008pi}.

\subsection{Establishing the Theorem}\label{the theorem}

First of all, we turn the manifold $Z=D\left(\lambda_{\HH}'\right)\oplus D\left(\lambda_{\HH}\right)$ described in the introduction into a $(U,fr)^{2}$--manifold: Its tangent bundle already decomposes into
\begin{equation*}
TZ\cong \pi^{*}TB\oplus\pi^{*}\lambda_{\HH}'\oplus\pi^{*}\lambda_{\HH}
\end{equation*}
where $\pi\colon Z\rightarrow B$ is the natural projection; dropping the first summand (which is stably trivial), we put $E_{1}=\pi^{*}\lambda_{\HH}'$, $E_{2}=\pi^{*}\lambda_{\HH}$ as complex vector bundles, using quaternion multiplication to trivialize the restriction of $E_{1}$ to the face $\partial_{1}Z=S\left(\lambda_{\HH}'\right)\oplus D\left(\lambda_{\HH}\right)$ and the restriction of $E_{2}$ to the other face. 

Next, we observe that the Hirzebruch elliptic genus of a quaternionic line is remarkably close to being modular w.r.t.\ the full modular group:

\begin{lem*} For a quaternionic line $\lambda_{\HH}$ we have{\textup{:}}
\begin{equation}\label{elli_quaternion_general}
Ell^{\Gamma_1\left(N\right)}\left(\lambda_{\HH}\right)=1+g_{2}^{(N)}c_2\left(\lambda_{\HH}\right)+\sum_{k\geq2}(-1)^kG_{2k}\frac{c_2^k\left(\lambda_{\HH}\right)}{\left(2k-2\right)!/2},
\end{equation}
where $g_{2}^{(N)}=\wp\left(\tau,{\textstyle\frac{2\pi i}{N}}\right)$ is a modular form of level $N$ and weight two, and it satisfies $g_{2}^{(N)}\equiv{\textstyle\frac{1}{12}}\mod\mathbb{Z}^{\Gamma}[\![q]\!]$.
\end{lem*}
\begin{proof}
Recall that the power series associated to Hirzebruch elliptic genus of level $N$ may be expressed as (see e.g.\ \cite[Appendix I]{Hirzebruch:1992dw})
$$Ell^{\Gamma_{1}\left(N\right)}(x)=x\frac{\Phi\left(\tau,x-\omega\right)}{\Phi\left(\tau,x\right)\Phi\left(\tau,-\omega\right)}$$
for $\omega=\frac{2\pi i}{N}$; here, the $\Phi$--function may be defined by $$\Phi(\tau,z)=2\sinh(z/2)\prod_{n\geq1}\frac{\left(1-e^{z}q^{n}\right)\left(1-e^{-z}q^{n}\right)}{\left(1-q^{n}\right)^{2}},$$ and it satisfies the transformation property
$$\Phi\left(\tau,x+2\pi i(\lambda\tau+\mu)\right)=q^{\lambda^{2}/2}e^{-\lambda x}(-1)^{\lambda+\mu}\Phi(\tau,x)\quad \forall\lambda,\mu\in\ZZ.$$
Thus, standard results for elliptic functions imply
\[Ell^{\Gamma_{1}\left(N\right)}(z)Ell^{\Gamma_{1}\left(N\right)}(-z)=z^{2}\left(\wp(\tau,z)-\wp(\tau,2\pi i/N)\right);\]
substituting $c_{2}\left(\lambda_{\HH}\right)=-z^{2}$ and using \[\wp\left(\tau,z\right)=\frac{1}{z^{2}}+2\sum_{k\geq2}G_{2k}(\tau)\frac{z^{2k-2}}{(2k-2)!},\]
equation \eqref{elli_quaternion_general} follows. Finally, the modular properties of $\wp\left(\tau,2\pi i/N\right)$ are an immediate consequence of those of the Hirzebruch elliptic genus, and we have the well-known expansion
\[\wp\left(\tau,2\pi i/N\right)=\frac{1}{12}+\frac{\zeta_{N}}{(1-\zeta_{N})^{2}}+\sum_{n\geq1}\left(\sum_{d|n}d(\zeta_{N}^{d}+\zeta_{N}^{-d}-2)\right)q^{n},\]
where $\zeta_{N}=\exp(2\pi i/N)$.
\end{proof}

The remaining steps in establishing the Theorem are then fairly standard:
\begin{proof}[Proof of the Theorem] Identifying the pair
$$\left(Z,\partial Z\right)=\left(D(\lambda_{\HH}')\oplus D(\lambda_{\HH}),S(\lambda_{\HH}')\oplus D(\lambda_{\HH})\cup D(\lambda_{\HH}')\oplus S(\lambda_{\HH})\right)$$
with the Thom space of the Whitney sum bundle $\lambda_{\HH}'\oplus\lambda_{\HH}$, we have:
\begin{equation*}
\begin{split}
{f}\left[S^{\oplus2}_{\HH}B\right]&\equiv\left\langle (Ell^{\Gamma}(\pi^{*}\lambda_{\HH}')-1)(Ell^{\Gamma}_{0}(\pi^{*}\lambda_{\HH})-1),[Z,\partial Z]\right\rangle\\
&=\left\langle \frac{Ell^{\Gamma}(\lambda_{\HH}')-1}{c_{2}(\lambda_{\HH}')}\frac{Ell^{\Gamma}_{0}(\lambda_{\HH})-1}{c_{2}(\lambda_{\HH})},[B]\right\rangle
\end{split}
\end{equation*}
Clearly, this expression vanishes if the dimension of the base is not divisible by four; therefore, in what follows, we put $\dim B=2m=4n>0$. Index theory then yields the  divisibility result (which is improvable for even $n$)
\begin{equation*}
\left\langle(-1)^{n}\frac{c_{2}^{n}\left(\lambda_{\HH}'\right)}{\left(2n\right)!/2},[B]\right\rangle=\left\langle\hat{A}(TB)\,ch\left(\lambda_{\HH}'\right),[B]\right\rangle\in\ZZ;
\end{equation*}
thus, combined with the congruence $d^{3}\equiv d^{5}\mod24$, the contribution to $f$ coming from $\omega^{n}=c_{2}^{n}(\lambda_{\HH}')$ is congruent to a modular form of top weight (i.e.\ of weight $2n+4$), hence vanishes. Furthermore, modulo the indeterminacy we have $0\equiv(g_{2}^{(N)}-\frac{1}{12})\frac{B_{2n+2}}{4n+4}(E_{2n+2}-1)\equiv-\frac{B_{2n+2}}{48(n+1)}(12g_{2}^{(N)}+E_{2n+2})$, so the previous argument also disposes of the contribution coming from $c_{2}^{n}(\lambda_{\HH})$.
\end{proof}

\begin{rmk*}
The absence of the `extremal' summands in \eqref{the formula} implies that the $f$--invariant of a double transfer on a base of positive dimension vanishes if the class of the base maps trivially along
\begin{equation*}
\Omega^{fr}_{*}\left(\HH P^{\infty}\times\HH P^{\infty}\right)\rightarrow\Omega^{fr}_{*}\left(\HH P^{\infty}\wedge\HH P^{\infty}\right)\rightarrow \Omega^{U}_{*}\left(\HH P^{\infty}\wedge\HH P^{\infty}\right);
\end{equation*}
subtracting products if necessary, we may also pass to the reduced versions.
\end{rmk*}

\begin{rmk*}
A computation along the lines of the preceding lemma shows
$$Td\left(\lambda_{\HH}\right)=1+\sum_{k=1}\left(-1\right)^{k+1}\frac{B_{2k}}{2k}\frac{c_2^k\left(\lambda_{\HH}\right)}{\left(2k-2\right)!},$$
hence an application of the Thom isomorphism to the disk bundle leads to the correct formula for the $e$--invariant of the single quaternionic transfer on a framed manifold of dimension $4n$, viz.\
\begin{equation}\label{corrected e}
e_{\CC}[S_{\HH}B]\equiv\frac{B_{2n+2}}{4n+4}\left\langle(-1)^{n}\frac{c_{2}^{n}\left(\lambda_{\HH}\right)}{\left(2n\right)!/2},[B]\right\rangle\mod\ZZ,
\end{equation}
the bracket yielding the $\lambda_{\HH}$--twisted Dirac index (which is even if $n$ is even).
\end{rmk*}

\subsection{Calculations in low dimensions}\label{calculations}

In principle, one can combine computations in the ring of divided congruences with divisibility results for combinations of Chern numbers (obtained from index theory) in order to simplify the formula \eqref{the formula} in the generic situation; as an illustration of this approach, we show: 

\begin{prop*}\label{low dimensional transfers} In the notation of the Theorem, let the manifold $B$ be of dimension $0<2m\leq14$. Then the $f$--invariant \textup{(}at the level $N=3$\textup{)} of the transfer $\left[S^{\oplus2}_{\HH}B\right]$ is trivial unless
\begin{itemize}
\item[\textup{(i)}] $\dim B=8$ and $\langle \eta\omega,[B]\rangle$ is odd,
\item[\textup{(ii)}] $\dim B=12$ and $4\nmid\langle \eta^{2}\omega,[B]\rangle$.
\end{itemize}
\end{prop*}

\begin{proof}
According to the Theorem, it is sufficient to check the cases $m=2n$: If $n=1$ in \eqref{the formula}, the sum is empty; if $n=2$, the only contribution comes from the coefficient of $\langle\eta\omega,[B]\rangle$, viz.\ $\frac{1}{240}\frac{E_4-1}{240}\equiv\frac{1}{2}(\frac{E_{4}-1}{240})^{2}$ which has been identified with the $f$--invariant of $\beta_{4/4}=\alpha_{4/4}^2$ in \cite{Bodecker:2008pi} (cf.\ also \cite{Bodecker:2009fk}). The remaining case $n=3$ results in
\begin{equation*}
\begin{split}
f\left[S^{\oplus2}_{\HH}B\right]&\equiv\frac{1}{12}\left(\frac{1}{240}\frac{E_6-1}{504}\left\langle\eta\omega^2,[B]\right\rangle+\frac{1}{504}\frac{E_4-1}{240}\left\langle\eta^2\omega,[B]\right\rangle\right)\\
&\equiv\frac{1}{4}\left(\frac{E_6-1}{8}\frac{1}{16}-\frac{1}{8}\frac{E_4-1}{16}\right)\left\langle\eta^2\omega,[B]\right\rangle\mod\underline{\underline{D}}^{\Gamma_{1}(3)}_{10}
\end{split}
\end{equation*}
since the index interpretation of $\langle ch(\lambda_{\HH}-2)ch(\lambda_{\HH}'-2),[B]\rangle$ implies the divisibilty result $12|\langle\left(\eta^2\omega+\eta\omega^2\right),[B]\rangle$, and the congruence $d^{p}\equiv d\mod p$ allows the removal of the the primes $p>3$ from the denominators.

It remains to be shown that the coefficient of $\langle\eta^2\omega,[B]\rangle$ is  of order precisely four; this will follow immediately once we have established a stronger statement, viz.\ that it coincides with the image of $\beta_{4/2,2}$. To this end, we perform an admittedly tedious yet (hopefully) transparent calculation, first expressing the coefficient in terms of the Eisenstein series $E_{4}$:
\begin{equation*}
\begin{split}
&\frac{1}{4}\left(\frac{E_6-1}{8}\frac{1}{16}-\frac{1}{8}\frac{E_4-1}{16}\right)\\
\equiv&-\frac{1}{4}\frac{E_4-1}{16}\frac{E_6-1}{8}-\frac{1}{2}\left(\frac{1}{8}\frac{E_4-1}{16}\right)\\
\equiv&\frac{1}{4}\left(\frac{E_4-1}{16}\right)^2-\frac{1}{2}\left(\frac{1}{8}\frac{E_4-1}{16}\right)
\end{split}
\end{equation*}
In terms of $E_{1}^{\Gamma}$ and $E_{3}^{\Gamma}$, the first summand can be rewritten as follows
\begin{equation*}
\begin{split}
&\frac{1}{4}\left(\frac{E_4-1}{16}\right)^2\\
=&\frac{1}{4}\left[\frac{1}{4}\left(\frac{E_1^4-1}{8}\right)^2+\frac{1}{2}\frac{E_1^4-1}{8}\left(E_1^4-E_1E_3\right)+\frac{1}{4}\left(E_1^4-E_1E_3\right)^2\right]\\
=&\frac{1}{4}\left[\left(\frac{E_1^2-1}{4}\right)^4+\left(\frac{E_1^2-1}{4}\right)^3+\frac{1}{4}\left(\frac{E_1^2-1}{4}\right)^2+\right.\\
&+\left.\frac{1}{2}\frac{E_1^4-1}{8}\left(E_1^4-E_1E_3\right)+\frac{1}{4}\left(E_1^4-E_1E_3\right)^2\right]\\
\equiv&\frac{E_1^2}{4}\left[\left(\frac{E_1^2-1}{4}\right)^4+\left(\frac{E_1^2-1}{4}\right)^3+\frac{1}{4}\left(\frac{E_1^2-1}{4}\right)^2+\right.\\&+\left.\frac{1}{2}\frac{E_1^4-1}{8}\left(E_1^4-E_1E_3\right)+\frac{1}{4}\left(E_1^4-E_1E_3\right)^2\right]\\
\equiv&\frac{1}{4}\left(\frac{E_1^2-1}{4}\right)^4+\frac{1}{4}\left(\frac{E_1^2-1}{4}\right)^3+\frac{1}{16}E_1^2\left(\frac{E_1^2-1}{4}\right)^2-\frac{1}{32}E_1E_3\\
=&\frac{1}{4}\left(\frac{E_1^2-1}{4}\right)^4+\frac{1}{2}\left(\frac{E_1^2-1}{4}\right)^3+\frac{1}{16}\left(\frac{E_1^2-1}{4}\right)^2-\frac{1}{32}E_1E_3,
\end{split}
\end{equation*}
where the next-to-last step follows from:
\begin{equation*}
\begin{split}
\frac{1}{8}\frac{E_1^4-1}{8}\left(E_1^6-E_1^3E_3\right)&\equiv\frac{1}{64}\left(E_1^3E_3-1\right)-\frac{1}{16}\frac{E_1^6-1}{4}\\
&\equiv\frac{1}{64}\left(E_1^3E_3-1\right)+\frac{1}{4}\frac{E_4-1}{16}\\
&\equiv\frac{1}{64}\left(E_1^3E_3-1\right)\\
&\equiv-\frac{1}{4}\frac{E_4-1}{16}E_1^3E_3\equiv-\frac{1}{4}\frac{E_4-1}{16}E_1E_3\\
&\equiv\frac{1}{4}\frac{E_6-1}{8}E_1E_3\equiv-\frac{1}{32}E_1E_3
\end{split}
\end{equation*}
On the other hand, we have
\begin{equation*}
\begin{split}
\frac{1}{16}\frac{E_4-1}{16}=&\frac{1}{32}\left(\frac{E_1^4-1}{8}+E_1^4-E_1E_3\right)\\
\equiv&\frac{1}{32}\frac{E_1^4-1}{8}-\frac{1}{32}E_1E_3\\
\equiv&\frac{1}{32}\left(\frac{E_1^4-1}{8}-\frac{E_1^2-1}{4}\right)-\frac{1}{32}E_1E_3\\
=&\frac{1}{16}\left(\frac{E_{1}^{2}-1}{4}\right)^{2}-\frac{1}{32}E_{1}E_{3}.
\end{split}
\end{equation*}
Therefore, the coefficient of $\langle\eta^2\omega,[B]\rangle$ is congruent to $\frac{1}{4}(\frac{E_{1}^{2}-1}{4})^{4}+\frac{1}{2}(\frac{E_{1}^{2}-1}{4})^{3}$, which can be identified with $f(\beta_{4/2,2})$ by the results of \cite{Bodecker:2009fk}.
\end{proof}

\subsection{The examples}

Recall that the homogeneous space $Sp(n)/Sp(1)^{\times n}$ may be identified with the quaternionic flag manifold, i.e.\ the space of $n$--flags in $\HH^{n}$. In particular, it comes with $n$ tautological quaternionic lines $\lambda_{i}$, and the (total) Whitney sum of these lines is trivial. If the latter condition is expressed in terms of the Chern classes (and the structure of the fiber bundles $Sp(k)/Sp(1)^{\times k}\rightarrow\HH P^{k-1}$ is taken into account), one  recovers the well-known integral cohomology ring of the quaternionic flag manifold, viz.\
\begin{equation}\label{cohomology ring}
H^{*}\left(Sp(n)/Sp(1)^{\times n};\ZZ\right)\cong\ZZ\left[t_{1},\dots,t_{n}\right]/\left(S^{+}\left(t_{1},\dots,t_{n}\right)\right)
\end{equation}
where $t_{i}=c_{2}(\lambda_{i})$ and $S^{+}\left(t_{1},\dots,t_{n}\right)$ is the ideal generated by the symmetric polynomials (of positive degree) in these Chern classes. Furthermore,  by the results of \cite{Singhof:1986km},  the quaternionic flag manifolds are stably parallelizable, hence can be framed.

\begin{rmk*}
For definiteness, we may consider the stable tangential framing of $Sp(n)/Sp(1)^{\times n}$ that is induced by the identification
$$(n-1)\oplus Ad\left(Sp(n),Sp(1)^{\times n}\right)\cong Ad\left(SU(2n),Sp(n)\right)|_{Sp(1)^{\times n}},$$
where $Ad(G,H)$ denotes the isotropy representation $H\rightarrow\text{Aut}\left(T_{eH}G/H\right)$.
\end{rmk*}

\begin{proof}[Proof of Corollary \ref{non-trivial examples}] First of all, observe that the Theorem implies that the image of the double quaternionic transfer in the tenth stable stem has trivial $f$--invariant; actually, since $\pi^{s}_{10}S^{0}\cong\ZZ/3\oplus\ZZ/2$ is generated by $\beta_{1}^{(3)}$ and $\alpha_{1}\alpha_{5}$, the image of this transfer itself is necessarily trivial. Moreover, the element $\beta_{1}^{(3)}$ generates the only 3-primary information visible to the $f$--invariant in the range under consideration 
(see e.g.\ the tables in \cite{Ravenel:2004xh}), so all that remains to be checked is that the conditions obtained in Proposition \ref{low dimensional transfers} can be met: Clearly, for (i) we may take the cartesian product of two copies of the flag manifold $Sp(2)/Sp(1)^{\times2}\approx \mathbb{HP}^{1}$ equipped with the pullbacks of the respective tautological lines. Similarly, for (ii) we may take the flag manifold $Sp(3)/Sp(1)^{\times3}$ equipped with any pair of distinct tautological lines $\lambda_{i}$, $\lambda_{j}$; then by \eqref{cohomology ring} we know that, up to sign, $t_{i}t_{j}^{2}$ is dual to the fundamental class.
\end{proof}

\begin{proof}[Proof of Corollary \ref{vanishing result}] Working in the cohomology ring \eqref{cohomology ring} for a fixed $n$, the successive use of all the relations reveals that
\begin{equation*}
\begin{split}
t_{1}^{n}&=-t_{1}^{n-1}\left(t_{2}+\dots+t_{n}\right)=t_{1}^{n-2}\left(t_{2}t_{3}+\dots+t_{n-1}t_{n}\right)=\dots=\\
&=\left(-1\right)^{n-1}t_{1}t_{2}\dots t_{n}=0;
\end{split}
\end{equation*}
specializing to $n=4$, we also have $t_{1}^{3}t_{2}^{3}=t_{1}^{3}t_{2}\left(t_{1}t_{3}+t_{1}t_{4}+t_{3}t_{4}\right)=0$.
Therefore, if $n\geq4$, no non-trivial summands occur in the formula \eqref{the formula} for the $f$--invariant of the tautological transfer of the quaternionic flag manifolds. 
\end{proof}

\bibliography{refbib_edited}
\end{document}